\documentclass[12pt]{amsart}
\usepackage{bbm}
\usepackage[all,poly,knot]{xy}
\usepackage{amscd,amsfonts,amsmath,mathrsfs,amssymb,amsthm,latexsym}

\newtheorem{theorem}{Theorem}[section]

\newtheorem{conjecture}[theorem]{Conjecture}
\newtheorem{lemma}[theorem]{Lemma}
\newtheorem{corollary}[theorem]{Corollary}
\newtheorem{remark}[theorem]{Remark}

\newtheorem{${}$}[theorem]{${}$}

\begin{document}

\title{Algebraic cycles on a generalized Kummer variety}

\author{Ze Xu}

\date{June 13, 2015}

\begin{abstract}
We compute explicitly the Chow motive of any generalized Kummer variety associated to any abelian surface. In fact, it lies in the rigid tensor subcategory of the category of Chow motives generated by the Chow motive of the underlying abelian surface. One application of this calculation is to show that the Hodge conjecture holds for arbitrary products of generalized Kummer varieties. As another application, all numerically trivial 1-cycles on arbitrary products of generalized Kummer varieties are smash-nipotent.
\end{abstract}
\maketitle

\bigskip

\section{Introduction}
\bigskip
This is part of the program to study algebraic cycles on projective hyperk\"{a}hler varieties. A smooth connected complex projective variety $X$ is \emph{hyperk\"{a}hler} if it is simply connected and $H^{2,0}(X)$ is generated by a nowhere vanishing holomorphic 2-form of $X$. This kind of variety is one of the three building blocks of smooth projective varieties with trivial canonical bundle due to Beauville-Bogomolov. Basic examples of projective hyperk\"{a}hler varieties are the Hilbert schemes of closed subschemes of finite length on projective K3 surfaces, generalized Kummer varieties, their projective deformations \cite{Bea83} and some sporadic examples.

This paper is devoted to the study of algebraic cycles on generalized Kummer varieties.

Let $A$ be a complex abelian surface. For any positive integer $n$, let $A^{(n)}$ be the $n$-th symmetric product of $A$ and $A^{[n]}$ be the Hilbert scheme of closed subschemes of length $n$ on $A$. Let $\pi: A^{[n]}\rightarrow A^{(n)}$ be the Hilbert-Chow morphism and $\sigma: A^{(n)}\rightarrow A$ be the addition morphism. Let $\alpha:=\sigma\circ\pi: A^{[n]}\rightarrow A$ be the composition morphism. It is clear that $\alpha$ is an isotrivial fibration. In fact, after the base change by $\textbf{n}: A\rightarrow A$, it is a trivial fibration. If $n\geq 2$, then the fiber $K^{[n]}:=\alpha^{-1}(o_{A})$ is a projective hyperk\"{a}hler variety of dimension $2(n-1)$, called the $n$-th \emph{generalized Kummer variety} associated to $A$ \cite{Bea83}. In fact, $K^{[2]}$ is just the Kummer K3 surface associated to $A$.

The first main result of this paper is the calculation of the Chow motive of a generalized Kummer variety. To state the result, we introduce some notations. Let $P(n)$ be the set of partitions of $n$. For any partition $\lambda=(\lambda_{1},\ldots,\lambda_{\ell_{\lambda}})=1^{a_{1}}2^{a_{2}}\cdots r^{a_{r}}$, let $\ell_{\lambda}$ be its length and $a_{i}$ be the appearing time of $i$. For any $\lambda\in P(n)$, define $e(\lambda):=gcd\{\lambda_{1},\ldots,\lambda_{\ell_{\lambda}}\}$. For any positive integer $m$, denote by $A[m]$ the group of $m$-torsion points of $A$. The starting point of our calculation is the construction of a natural stratification for the projective surjective semi-small morphism $\pi_{0}: K^{[n]}\rightarrow K^{(n)}:=\sigma^{-1}(o_{A})$ induced by $\pi$: $K^{(n)}=\coprod_{\lambda\in P(n)}\coprod_{\tau\in A[e(\lambda)]}K_{\lambda,\tau}^{(n)}$. For the closure $\overline{K}_{\lambda,\tau}^{(n)}$ of each stratum $K_{\lambda,\tau}^{(n)}$, there is a normalization morphism $K_{\tau}^{(\lambda)}\rightarrow\overline{K}_{\lambda,\tau}^{(n)}$. Here $K_{\tau}^{(\lambda)}$ is a torsion translation of $K_{o}^{(\lambda)}$, a quotient of an abelian variety. For a precise description of $K_{\tau}^{(\lambda)}$, see Section 2.

\begin{theorem}\label{thm1.1}
Let $A$ be a complex abelian surface and $K^{[n]}$ be its associated $n$-th generalized Kummer variety. Then in the category of Chow motives, there is a canonical isomorphism
$$
h(K^{[n]})\cong\bigoplus_{\lambda\in P(n)}\bigoplus_{\tau\in A[e(\lambda)]}h(K_{\tau}^{(\lambda)})\otimes\mathbb{L}^{n-\ell(\lambda)}.
$$
\end{theorem}

Roughly speaking, the Chow motive of a generalized Kummer variety is identified with a direct sum of proper Tate shifts of the Chow motives of quotients of projective commutative algebraic subgroup of self-products of the underlying abelian surface. The isomorphism in Theorem \ref{thm1.1} is realized by explicitly constructing the decomposition of the diagonal $\Delta_{K^{[n]}}$ into a sum of orthogonal idempotents, each of which corresponds to a direct summand. A key point for the proof is to consider the topological realization of the correspondences and to make use of the classical decomposition theorem.

It turns out that the Chow motive of the product $A\times K^{[n]}$ has a more elegant form.

\begin{corollary}
There is a canonical isomorphism of Chow motives
$$
h(A\times K^{[n]})\cong\bigoplus_{\lambda\in P(n)}\bigoplus_{\tau\in A[e(\lambda)]}h(A^{(\lambda)})\otimes\mathbb{L}^{n-\ell_{\lambda}},
$$
where $\mathbb{L}$ is the Lefschetz motive.
\end{corollary}

This corollary refines a formula in \cite{GS93} at the level of Hodge structures.

We provide two applications of the above results to the study of algebraic cycles on generalized Kummer varieties. The first one is to verify the Hodge conjecture for products of generalized Kummer varieies, by making use of the motivation principle for Hodge cycles proved in \cite{Ar06}.

\begin{theorem}
Let $A_{i} (1\leq i\leq r)$ be complex abelian surfaces and $K_{i}^{[n_{i}]}$ be the $n_{i}$-th generalized Kummer variety associated to $A_{i}$. Then the Hodge conjecture holds for the product $K_{1}^{[n_{1}]}\times\cdots\times K_{r}^{[n_{r}]}$.
\end{theorem}

The other one is to show that all numerically trivial 1-cycles on arbitrary products of generalized Kummer varieties are smash-nilpotent. This is achieved by combining Theorem \ref{thm1.1} with recent results on 1-cycles on abelian varieties in \cite{Seb13}.

\begin{theorem}
Let $A_{i} (1\leq i\leq r)$ be abelian surfaces and $K_{i}^{[n_{i}]}$ be the $n_{i}$-th generalized Kummer variety associated to $A_{i}$. Then any numerically trivial 1-cycle on the product $K_{1}^{[n_{1}]}\times\cdots\times K_{r}^{[n_{r}]}$ is smash-nilpotent.
\end{theorem}

\emph{Convention}: We work over the field of complex numbers $\mathbb{C}$ unless otherwise stated. We will freely use notations and results from intersection theory. See \cite{F98}.

\medskip\noindent
{\it Acknowledgements}: The author would like to thank Professor Chenyang Xu, Professor Baohua Fu and Professor Kejian Xu gratefully for supports and encouragements. 

This work is partially supported by National Natural Science Foundation of China (No.11201454).

\bigskip

\section{The Chow motive of a generalized Kummer variety}
\bigskip

In this section, we carry out explicit calculations of the Chow motives of generalized Kummer varieties associated to abelian surfaces. The main result is Theorem \ref{thm2.7}.

\subsection{A natural stratification of $K^{[n]}$}

We recollect the basic geometric setup of generalized Kummer varieties and essential combinatorial information involved.

Let $A$ be an abelian surface. The $n$-th symmetric product $A^{(n)}$ has a natural stratification $A^{(n)}=\coprod_{\lambda\in P(n)}A_{\lambda}^{(n)}$, where the stratum $A_{\lambda}^{(n)}$ is the irreducible locally closed smooth subvariety of $A^{(n)}$ consisting of points of the form $\lambda_{1}x_{1}+\cdots+\lambda_{\ell_{\lambda}}x_{\ell_{\lambda}}$ with $x_{i}\in A$ and $x_{i}\neq x_{j}$ for every $i\neq j$. Let $\overline{A}_{\lambda}^{(n)}$ be the closure of the stratum $A_{\lambda}^{(n)}$. Let $A^{(\lambda)}:=A^{(a_{1})}\times\cdots\times A^{(a_{r})}$ be the quotient of $A^{\ell_{\lambda}}$ by the natural action of the finite group $\mathfrak{G}_{\lambda}:=\mathfrak{G}_{a_{1}}\times\cdots\times\mathfrak{G}_{a_{r}}$. Then there is a natural $\mathfrak{G}_{\lambda}$-invariant morphism $\lambda: A^{\lambda}\rightarrow A^{(n)}$ with image $\overline{A}_{\lambda}^{(n)}$. This descends to the normalization morphism of the image $\lambda: A^{(\lambda)}\rightarrow A^{(n)}$, denoted by the same symbol. Let $A_{\lambda}^{[n]}$ be the reduced scheme of $\pi^{-1}(A_{\lambda}^{(n)})$. Then $A^{[n]}$ has a natural stratification $A^{[n]}=\coprod_{\lambda\in P(n)}A_{\lambda}^{[n]}$. Let $\overline{A}_{\lambda}^{[n]}$ be the closure of the stratum $A_{\lambda}^{[n]}$. Then $\overline{A}_{\lambda}^{[n]}=(\pi^{-1}(\overline{A}_{\lambda}^{(n)}))_{\text{red}}$. A crucial fact for this stratification is that for each $\lambda\in P(n)$, the morphism $\pi: A_{\lambda}^{[n]}\rightarrow A_{\lambda}^{(n)}$ is a Zariski locally trivial fibration with irreducible fiber admitting a cellular decomposition \cite{deCM02}.

Let $K^{(n)}:=\sigma^{-1}(o_{A})$. Then it is an irreducible quotient variety and hence normal. Let $j: K^{(n)}\rightarrow A^{(n)}$ be the natural closed immersion. Then we have the following commutative diagram with each square cartesian
$$
\xymatrix{
K^{[n]} \ar[d]_{j} \ar[r]^{\pi_{0}} & K^{(n)} \ar[d]_{} \ar[r]^{\sigma} & o_{A} \ar[d]^{} \\
A^{[n]} \ar[r]^{\pi} & A^{(n)} \ar[r]^{\sigma} & A.}
$$
The stratification of $A^{(n)}$ given above induces a decomposition of $K^{(n)}$ into a disjoint union of locally closed smooth subvarieties:
$$
K^{(n)}=\coprod_{\lambda\in P(n)}K_{\lambda}^{(n)},
$$
where $K_{\lambda}^{(n)}=A_{\lambda}^{(n)}\cap K^{(n)}$ is the scheme-theoretic intersection.  

A feature of the above decomposition is that not all the pieces $K_{\lambda}^{(n)}$ are irreducible, although they are all smooth. Now we want to give all the irreducible components of $K_{\lambda}^{(n)}$. For any $\lambda\in P(n)$, define $e(\lambda):=gcd\{\lambda_{1},\ldots,\lambda_{\ell_{\lambda}}\}$. Now define a homomorphism of abelian varieties
$s: A^{a_{1}}\times\cdots\times A^{a_{r}}\rightarrow A$
given by
$$
(x_{1},\ldots, x_{\ell_{\lambda}})\mapsto\sum_{i=1}^{\ell_{\lambda}}\lambda_{i}x_{i}.
$$
Let $K^{\lambda}=\text{Ker}(s)$. Then $K^{\lambda}$ is a smooth projective commutative algebraic group with exactly $e(\lambda)^{4}$ irreducible components. Then
$$K^{\lambda}=\coprod_{\tau\in A[e(\lambda)]}K_{\tau}^{\lambda},$$
where
$K_{\tau}^{\lambda}=\{(x_{1},\ldots,x_{\ell_{\lambda}})\in A^{\ell_{\lambda}}|\sum_{i=1}^{\ell_{\lambda}}\frac{\lambda_{i}}{e(\lambda)}x_{i}=\tau\}$.
Note that $K_{o}^{\lambda}$ is the connected component of the identity of $K^{\lambda}$.  For any chosen $\tau'\in A[n]$ such that $\frac{n}{e(\lambda)}\cdot\tau'=\tau$, the translation morphism $t_{\tau'}$ on $A$ induces the translation $t_{\tau'}$ on $A^{\ell_{\lambda}}$ mapping $(x_{1},\ldots,x_{\ell_{\lambda}})$ to $(x_{1}+\tau',\ldots,x_{\ell_{\lambda}}+\tau')$, which restricts to $K^{\lambda}$ and maps $K_{o}^{\lambda}$ isomorphically to $K_{\tau}^{\lambda}$.
The action of $\mathfrak{G}_{\lambda}$ on $A^{\lambda}$ restricts to $K^{\lambda}$ and this induced action is compatible with the irreducible components decomposition, that is, $\mathfrak{G}_{\lambda}$ acts on each $K_{\tau}^{\lambda}$ with quotient, denoted by, $K_{\tau}^{(\lambda)}$. Then one has a natural decomposition
$$
K^{(\lambda)}=\coprod_{\tau\in A[e(\lambda)]}K_{\tau}^{(\lambda)}.
$$
The natural morphism $\lambda: A^{(\lambda)}\rightarrow A^{(n)}$ induces the morphism $\lambda: K^{(\lambda)}\rightarrow K^{(n)}$. The image of $K_{\tau}^{(\lambda)}$ under $\lambda$ is denoted by $\overline{K}_{\lambda,\tau}^{(n)}$. The induced morphism $\lambda_{\tau}: K_{\tau}^{(\lambda)}\rightarrow \overline{K}_{\lambda,\tau}^{(n)}$ is the normalization. Then we have a commutative diagram
$$
\xymatrix{
  K_{\tau}^{(\lambda)} \ar[d]_{} \ar[r]^{\lambda} & \overline{K}_{\lambda,\tau}^{(n)} \ar[d]_{} \ar[r]^{} & K^{(n)} \ar[d]^{} \\
  A^{\lambda} \ar[r]^{} & \overline{A}_{\lambda}^{(n)} \ar[r]^{} & A^{(n)}.}
$$
Let $K_{\lambda,\tau}^{(n)}=A_{\lambda}^{(n)}\cap\overline{K}_{\lambda,\tau}^{(n)}$. Then the irreducible components of $K_{\lambda}^{(n)}$ are exactly $K_{\lambda,\tau}^{(n)}$ with each $\tau\in A[e(\lambda)]$. Therefore, we have a natural stratification of $K^{(n)}$ as follows:
$$
K^{(n)}=\coprod_{\lambda\in P(n)}\coprod_{\tau\in A[e(\lambda)]}K_{\lambda,\tau}^{(n)}.
$$
Let $\pi_{0}: K^{[n]}\rightarrow K^{(n)}$ be the (restricted) Hilbert-Chow morphism. Let $K_{\lambda,\tau}^{[n]}$ be the underlying reduced scheme of $\pi_{0}^{-1}(K_{\lambda,\tau}^{(n)})$. Then the morphism $\pi_{0}: K_{\lambda,\tau}^{[n]}\rightarrow K_{\lambda,\tau}^{(n)}$ is a Zariski locally trivial fibration, since it is a base change of $\pi: A_{\lambda}^{[n]}\rightarrow A_{\lambda}^{(n)}$. Let $\overline{K}_{\lambda,\tau}^{[n]}$ be the closure of $K_{\lambda,\tau}^{[n]}$. Then $\overline{K}_{\lambda,\tau}^{[n]}$ coincides with the underlying reduced scheme of $\pi_{0}^{-1}(\overline{K}_{\lambda,\tau}^{(n)})$. Finally, we reach at a natural stratification of $K^{[n]}$:
$$
K^{[n]}=\coprod_{\lambda\in P(n)}\coprod_{\tau\in A[e(\lambda)]}K_{\lambda,\tau}^{[n]}.
$$

Note that there is a natural partial order on the set of partitions. For $\lambda, \mu\in P(n)$, $\lambda\geq\mu$ if there is a decomposition $\{I_{1},\ldots,I_{\ell_{\mu}}\}$ of the set $\{1,\ldots,\ell_{\lambda}\}$ such that $\mu_{1}=\sum_{i\in I_{1}}\lambda_{i}$, \ldots, $\mu_{\lambda}=\sum_{i\in I_{\ell_{\mu}}}\lambda_{i}$. Note that $\lambda\geq\mu$ if and only if $A_{\mu}^{(n)}\subseteq \overline{A}_{\lambda}^{(n)}$. So, for any $\lambda\in P(n)$, $A_{\geq\lambda}^{(n)}:=\coprod_{\mu\geq\lambda}A_{\mu}^{(n)}$ is an open subset of $A^{(n)}$ and $A_{\lambda}^{(n)}$ is a closed subset of $A_{\geq\lambda}^{(n)}$. Then we have an open subset $A_{\geq\lambda}^{[n]}$ of $A^{[n]}$ by base change followed by reduction. Therefore, we have open subsets $K_{\geq\lambda}^{(n)}:=A_{\geq\lambda}^{(n)}\cap K^{(n)}$ of $K^{(n)}$ and $K_{\geq\lambda}^{[n]}:=A_{\geq\lambda}^{[n]}\cap K^{[n]}$ of $K^{[n]}$. Moreover, $K_{\lambda,\tau}^{(n)}$ and $K_{\lambda,\tau}^{[n]}$ are closed subvarieties of $K_{\geq\lambda}^{(n)}$ and $K_{\geq\lambda}^{[n]}$ of dimension $n-2+\ell_{\lambda}$ respectively. For any close point $z\in K_{\tau,\text{sm}}^{(\lambda)}\cong K_{\lambda,\tau}^{(n)}$, let $F_{\lambda,\tau}=(\pi_{0}^{-1}(z))_{\text{red}}\subseteq K_{\lambda,\tau}^{[n]}\subseteq K_{\geq\lambda}^{[n]}$, which is closed of dimension $n-\ell_{\lambda}$.

\subsection{Explicit calculations}

\begin{lemma}\label{lem2.1}
For any partition $\lambda\in P(n)$ and $\tau\in A[e(\lambda)]$, the intersection number
$$
\int_{K_{\geq\lambda}^{[n]}}[F_{\lambda,\tau}]\cdot [K_{\lambda,\tau}^{[n]}]
$$
is nonzero.
\end{lemma}
In fact, this is an immediate consequence of the decomposition theorem due to Beilinson, Bernstein and Deligne \cite{BBD82}, recalling that the morphism $\pi_{0}: K^{[n]}\rightarrow K^{(n)}$ is semi-small, actually a crepant resolution.

Here we only give the statement of the decomposition theorem for semi-small morphism, in which case the theorem turns out to be simple and explicit. Now we recall some relevant notations. A proper surjective morphism $f: X\rightarrow S$ of complex algebraic varieties is semi-small (resp. small) if for $\delta\in\mathbb{N}$, defining $S_{f}^{\delta}:=\{s\in S|\text{dim}f^{-1}(s)=\delta\}$, then $\text{dim}S-\text{dim}S_{f}^{\delta}\geq 2\delta$ for any $\delta\geq 0$ (resp. $\text{dim}S-\text{dim}S_{f}^{\delta}>2\delta$ for any $\delta>0$); equivalently, the irreducible components of $X\times_{S}X$ have dimension at most $\text{dim}X$ (resp., the $\text{dim}X$-dimensional irreducible components of $X\times_{S}X$ dominates $S$). For any semi-small morphism $f: X\rightarrow S$, there exists a stratification: $S=\coprod_{\alpha\in\Xi}S_{\alpha}$, where each $S_{\alpha}$ is a locally closed smooth subvariety of $S$ and the induced morphism $f^{-1}(S_{\alpha})\rightarrow S_{\alpha}$ is a locally topologically trivial fibration. We say $\Omega:=\{\alpha\in\Xi|\text{dim}S-\text{dim}S_{\alpha}=2\text{dim}f^{-1}(s)\ \text{for any}\ s\in S_{\alpha}\}$ is the set of relevant strata for $f$. For examples, $\pi: A^{[n]}\rightarrow A^{(n)}$ and $\pi_{0}: K^{[n]}\rightarrow K^{(n)}$ are both semi-small with stratifications $A^{(n)}=\coprod_{\lambda\in P(n)}$ and $K^{(n)}=\coprod_{\lambda\in P(n)}\coprod_{\tau\in A[e(\lambda)]}K_{\lambda,\tau}^{(n)}$ respectively and in both cases, all strata are relevant. For a complex algebraic variety $X$ and a local system $\mathcal{L}$ on some Zariski open subset of its smooth locus $X_{\text{sm}}$, we will denote by $IC_{X}(\mathcal{L})$ the intersection cohomology complex associated to $\mathcal{L}$.

\begin{theorem}\label{thm2.2}(\cite{BBD82})
Let $X$ be a quotient variety of dimension $d$ and $f: X\rightarrow S$ be a proper surjective semi-small morphism of complex algebraic varieties. Then there is a canonical isomorphism
$$
Rf_{*}\mathbb{Q}_{X}[d]\cong\bigoplus_{\alpha\in\Omega}IC_{\overline{S}_{\alpha}}(\mathcal{L}_{\alpha})
$$
in the bounded derived category $D_{cc}^{b}(S)$ of constructible complexes of sheaves of $\mathbb{Q}$-vector spaces on $S$, where $\Omega$ is the set of relevant strata for $f$ and each $\mathcal{L}_{\alpha}$ is the semi-simple local system on $S_{\alpha}$ defined by the monodromy action on the maximal dimensional irreducible components of the fibers of $f$ over $S_{\alpha}$.
\end{theorem}
\begin{remark}
For a relevant stratum $S_{\alpha}$ and any $s_{\alpha}\in S_{\alpha}$, then $2\text{dim}f^{-1}(s_{\alpha})=d-\text{dim}S$ by definition. Take a contractible analytic neighborhood $U$ of $s_{\alpha}$ and $U':=f^{-1}(U)$. Let $\{Y_{i}\}$ be the set of maximal $\frac{1}{2}(d+\text{dim}S_{\alpha})$-dimensional irreducible components of $f^{-1}(S_{\alpha}\cap U)$ and $\{F_{i}\}$ be the set of maximal $\frac{1}{2}(d-\text{dim}S_{\alpha})$-dimensional irreducible components of $f^{-1}(s_{\alpha})$. By intersection theory \cite{F98}, there are well-defined $\mathbb{Q}$-valued (refined) intersection numbers $\int_{U'}F_{i}\cdot Y_{j}$. These numbers are independent of the choice of $s_{\alpha}$ and $U$, and are monodromy invariant. Then we get an intersection form associated to any relevant stratum. The meaning of Theorem \ref{thm2.2} is that the intersection form associated to each relevant stratum is nondegenerate.

As we have seen, the morphism $\pi_{0}: K^{[n]}\rightarrow K^{(n)}$ is a semi-small. All strata are relevant and for each stratum $K_{\lambda,\tau}^{(n)}$ and $z\in K_{\lambda,\tau}^{(n)}$, the fiber $\pi_{0}^{-1}(z)$ is irreducible with a cellular structure and $\pi_{0}^{-1}(K_{\lambda,\tau}^{(n)})\cap U$ is irreducible for any contractible analytic neighborhood $U$ of $z$. The intersection form associated to the stratum $K_{\lambda,\tau}^{(n)}$ is determined by the single intersection number $\int_{K_{\geq\lambda}^{[n]}}[F_{\lambda,\tau}]\cdot [K_{\lambda,\tau}^{[n]}]$, which is nonzero by the decomposition theorem.
\end{remark}

Now we construct involved correspondences for our computations of Chow motives. Let $\Theta_{\tau}^{\lambda}$ be the underlying reduced scheme of $K_{\tau}^{\lambda}\times_{K^{(n)}}K^{[n]}$. Since the correspondence $\Theta_{\tau}^{\lambda}$ is invariant under the action of $\mathfrak{G}_{\lambda}$ on the first factor of the product, one can define $\widehat{\Theta}_{\tau}^{\lambda}:=\Theta_{\tau}^{\lambda}/\mathfrak{G}_{\lambda}=(K_{\tau}^{(\lambda)}\times_{K^{(n)}}K^{[n]})_{\text{red}}$.

\begin{lemma}\label{lem2.4}
For $\lambda\in P(n)$, $\tau\in A[e(\lambda)]$ and $\lambda'\in P(n)$, $\tau'\in A[e(\lambda')]$, one has as correspondences
$$
^{t}\widehat{\Theta}_{\tau}^{\lambda}\circ\widehat{\Theta}_{\tau'}^{\lambda'}=\delta_{(\lambda,\tau)(\lambda',\tau')}d_{\lambda,\tau}\Delta_{K_{\tau}^{(\lambda)}}\ \text{in}\  \text{CH}^{\ell_{\lambda}+\ell_{\lambda'}}(K_{\tau'}^{(\lambda')}\times K_{\tau}^{(\lambda)})_{\mathbb{Q}},
$$
where $d_{\lambda,\tau}:=\int_{K^{[n]}}[F_{\lambda,\tau}]\cdot K_{\lambda,\tau}^{[n]}\neq 0$.
\end{lemma}
\begin{proof}
First, suppose that $\lambda\neq\lambda'$. Denote by $p_{ij}$ the projection from $K_{\tau'}^{(\lambda')}\times K^{n}\times K_{\tau}^{(\lambda)}$ to the product of the $i$-th and $j$-th factors. Then by definition of push forward of algebraic cycles, it suffices to show that
$$
\text{dim}p_{13}(p_{12}^{-1}(\widehat{\Theta}_{\tau'}^{\lambda'})\cap p_{23}^{-1}(^{t}\widehat{\Theta}_{\tau}^{\lambda}))<\ell_{\lambda}+\ell_{\lambda'}.
$$
In fact, $\text{dim}p_{13}(p_{12}^{-1}(\widehat{\Theta}_{\tau'}^{\lambda'})\cap p_{23}^{-1}(^{t}\widehat{\Theta}_{\tau}^{\lambda}))=\text{dim}\overline{K}_{\lambda,\tau}^{(n)}\cap\overline{K}_{\lambda',\tau'}^{(n)}
<2\text{min}\{\ell_{\lambda},\ell_{\lambda'}\}$ thanks to $\lambda\neq\lambda'$, as required.

 Now assume that $\lambda=\lambda'$ and $\tau\neq\tau'$. Then $^{t}\widehat{\Theta}_{\tau}^{\lambda}\circ\widehat{\Theta}_{\tau'}^{\lambda'}=0$, since $K_{\tau}^{(\lambda)}$ and $K_{\tau'}^{(\lambda)}$ are disjoint.

Finally, it is easy to see that the support of the algebraic cycle $^{t}\widehat{\Theta}_{\tau}^{\lambda}\circ\widehat{\Theta}_{\tau}^{\lambda}$ is the diagonal $\Delta_{K_{\tau}^{(\lambda)}}\subseteq K_{\tau}^{(\lambda)}\times K_{\tau}^{(\lambda)}$, irreducible of the expected dimension. So, $^{t}\widehat{\Theta}_{\tau}^{\lambda}\circ\widehat{\Theta}_{\tau}^{\lambda}=d_{\lambda,\tau}\Delta_{K_{\tau}^{(\lambda)}}$ for some $d_{\lambda,\tau}\in\mathbb{Q}$. For any closed point $z\in K_{\tau,\text{sm}}^{(\lambda)}$,  one has $\widehat{\Theta}_{\tau*}^{\lambda}([z])=[F_{\lambda,\tau}]$ in $\text{CH}^{*}(K^{[n]})_{\mathbb{Q}}$. Then by the projection formula, one has
$$
d_{\lambda,\tau}=\int_{K_{\tau}^{(\lambda)}}\ ^{t}\widehat{\Theta}_{\tau*}^{\lambda}\circ\widehat{\Theta}_{\tau*}^{\lambda}([z])
=\int_{K_{\tau}^{(\lambda)}}\ ^{t}\widehat{\Theta}_{\tau*}^{\lambda}([F_{\lambda,\tau}])=\int_{K^{[n]}}[F_{\lambda,\tau}]\cdot [K_{\lambda,\tau}^{[n]}]\neq 0,
$$
according to Lemma \ref{lem2.1}.
\end{proof}
\begin{lemma}\label{lem2.5}
Let $\Delta_{\lambda,\tau}=\frac{1}{d_{\lambda,\tau}}\widehat{\Theta}_{\tau}^{\lambda}\circ ^{t}\widehat{\Theta}_{\tau}^{\lambda}$. Then the following hold in the correspondence ring $\text{CH}^{2n-2}(K^{[n]}\times K^{[n]})_{\mathbb{Q}}$:

(i) $\Delta_{\lambda,\tau}\circ\Delta_{\lambda',\tau'}=\delta_{(\lambda,\tau)(\lambda',\tau')}\Delta_{\lambda,\tau}$;

(ii) $\Delta_{K^{[n]}}=\sum_{\lambda\in P(n)}\sum_{\tau\in A[e(\lambda)]}\Delta_{\lambda,\tau}$.
\end{lemma}
\begin{proof}
(i) This follows immediately from Lemma \ref{lem2.4}.

(ii) We draw the conclusion by passing to the topological realizations of the corrspondences to Borel-Moore homology according to the following Lemma \ref{lem2.6}.

First, noting that $K_{\tau}^{(\lambda)}$ is a quotient variety, then $IC_{K_{\tau}^{(\lambda)}}(\mathbb{Q}_{K_{\tau,\text{sm}}^{(\lambda)}})\cong\mathbb{Q}_{K_{\tau}^{(\lambda)}}[2\ell_{\lambda}-2]$. Since the morphism $\lambda_{\tau}: K_{\tau}^{(\lambda)}\rightarrow \overline{K}_{\lambda,\tau}^{(n)}$ is the normalization, it is small and therefore, $Rf_{*}IC_{K_{\tau}^{(\lambda)}}(\mathbb{Q}_{K_{\tau,\text{sm}}^{(\lambda)}})\cong IC_{\overline{K}_{\lambda,\tau}^{(n)}}\mathcal{L}$, where $\mathcal{L}$ is the local system given by the monodromy action on the points of a general fiber of $\lambda_{\tau}$. Then by Theorem \ref{thm2.2}, one has
$$
R\pi_{0*}\mathbb{Q}_{K^{[n]}}[2n-2]\cong\bigoplus_{\lambda\in P(n)}\bigoplus_{\tau\in A[e(\lambda)]}R\lambda_{\tau*}\mathbb{Q}_{K_{\tau}^{(\lambda)}}[2\ell_{\lambda}-2].
$$
We have known that $\frac{1}{d_{\lambda,\tau}}^{t}\widehat{\Theta}_{\tau}^{\lambda}\circ\widehat{\Theta}_{\tau}^{\lambda}=\Delta_{K_{\tau}^{(\lambda)}}$ by Lemma \ref{lem2.4}. Therefore, the morphism on $R\pi_{0*}\mathbb{Q}_{K^{[n]}}[2n-2]$ given as the topological realization of the algebraic correspondence $\sum_{\lambda\in P(n)}\sum_{\tau\in A[e(\lambda)]}\Delta_{\lambda,\tau}$, is the identity. Now since the cycle $\sum_{\lambda\in P(n)}\sum_{\tau\in A[e(\lambda)]}\Delta_{\lambda,\tau}$ is supported on a closed subscheme of dimension $2n-2$ of $K^{[n]}\times_{K^{(n)}} K^{[n]}$ and its topological realization is the identity, it is just the diagonal $\Delta_{K^{[n]}}$.
\end{proof}

\begin{lemma}\label{lem2.6}(\cite{CH00})
Let $X$ and $X'$ be quotient varieties, and $f: X\rightarrow S$, $f': X'\rightarrow S$ be proper morphisms of analytic varieties. Then for any integers $i, j$ and $k$, there are canonical isomorphisms
$$
\rho_{X'X}: \text{Hom}_{D_{cc}^{b}(S)}(Rf'_{*}\mathbb{Q}_{X'}[i],Rf_{*}\mathbb{Q}_{X}[j])\cong H_{2\text{dim}S+i-j}^{BM}(X'\times_{S}X,\mathbb{Q}),
$$
where the right hand side is the $(2\text{dim}S+i-j)-$th Borel-Moore homology with $\mathbb{Q}$-coefficients. Moreover, for another morphism $f'': X''\rightarrow S$ as above, given morphisms $\theta: Rf'_{*}\mathbb{Q}_{X'}[i]\rightarrow Rf_{*}\mathbb{Q}_{X}[j]$ and $\eta: Rf_{*}\mathbb{Q}_{X}[j]\rightarrow Rf''_{*}\mathbb{Q}_{X''}[k]$, one has $\rho_{X'X''}(\eta\circ\theta)=\rho_{XX''}(\eta)\circ\rho_{X'X}(\theta)$, where the right hand side is the (refined) composition of homological correspondences.
\end{lemma}

Now our first main result as follows is a consequence of Lemma \ref{lem2.4} and Lemma \ref{lem2.5}.

\begin{theorem}\label{thm2.7}
Let $A$ be an abelian surface and $K^{[n]}$ be its associated $n$-th generalized Kummer variety. Then in the category of Chow motives, there is a canonical isomorphism
$$
h(K^{[n]})\cong\bigoplus_{\lambda\in P(n)}\bigoplus_{\tau\in A[e(\lambda)]}h(K_{\tau}^{(\lambda)})\otimes\mathbb{L}^{n-\ell(\lambda)}, \eqno (*)
$$
where $e(\lambda):=gcd\{\lambda_{1},\ldots,\lambda_{\ell_{\lambda}}\}$.
\end{theorem}
\begin{proof}
We have seen  by Lemma \ref{lem2.5} (ii) that
$$
\Delta_{K^{[n]}}=\sum_{\lambda\in P(n)}\sum_{\tau\in A[e(\lambda)]}\Delta_{\lambda,\tau}.
$$
This is equivalent to that there is an isomorphism of Chow motives
$$
h(K^{[n]})\cong\bigoplus_{\lambda\in P(n)}\bigoplus_{\tau\in A[e(\lambda)]}(K^{[n]},\Delta_{\lambda,\tau})
$$

On the other hand, according to Lemma \ref{lem2.4}, we have an isomorphism of Chow motives
$$
(K^{[n]},\Delta_{\lambda,\tau})\cong h(K_{\tau}^{(\lambda)})\otimes\mathbb{L}^{n-\ell_{\lambda}}
$$
for any $\lambda\in P(n)$ and $\tau\in A[e(\lambda)]$.
Combining the two isomorphisms above, we get the required isomorphism (*).
\end{proof}

It seems that the Chow motive of $K_{\tau}^{(\lambda)}$ is hard to identify. However, this is not the case. In fact, as we will see in the following corollary, the Chow motive of $K_{\tau}^{(\lambda)}$ is in fact closely related to the Chow motive of $A^{(\lambda)}$.

\begin{corollary}\label{cor2.8}
There is a canonical isomorphism of Chow motives
$$
h(A\times K^{[n]})\cong\bigoplus_{\lambda\in P(n)}\bigoplus_{\tau\in A[e(\lambda)]}h(A^{(\lambda)})\otimes\mathbb{L}^{n-\ell_{\lambda}}.
$$
\end{corollary}
\begin{proof}
It suffices to establish canonical isomorphisms of Chow motives
$$
h(A\times K_{\tau}^{(\lambda)})\cong h(A^{(\lambda)})
$$
for all $\lambda\in P(n)$ and $\tau\in A[e(\lambda)]$.  Now for each $\lambda\in P(n)$ and each $\tau\in A[e(\lambda)]$, there is an action of $A[\frac{n}{e(\lambda)}]$ on $A\times K_{\tau}^{(\lambda)}$ given by
$$
\sigma(x,z)=(x-\sigma, t_{\sigma}z)
$$
for $\sigma\in A[\frac{n}{e(\lambda)}]$ and $z\in K_{\tau}^{(\lambda)}$, where $t_{\sigma}$ is the translation by $\sigma$, which is easily seen to be well-defined on $K_{\tau}^{(\lambda)}$. There is a natural morphism $\phi: A\times K_{\tau}^{(\lambda)}\rightarrow A^{(\lambda)}$ defined by
$$
\phi(x,z)=t_{x}z.
$$
In fact, one can check that this morphism is just the quotient by the action of $A[\frac{n}{e(\lambda)}]$ defined previously. It remains to show that at the level of Chow motives, the morphism $\phi$ gives actually an isomorphism. For this point, we only have to note that any torsion translation action on an abelian variety (or any torsion translation of an abelian subvariety of a proper (nonconnected) commutative algebraic group) is trivial at the level of Chow groups and hence of Chow motives, which is a simple fact, see, for example, \cite{H88}.
\end{proof}

\begin{remark}
(i) It is shown by Voisin \cite{Voi14} that the Chow motives of many K3 surfaces are direct summands of abelian varieties, assuming the variational Hodge conjecture in degree 4. Then for such a K3 surface $S$, the Chow motive of $S^{[n]}$ is \emph{of abelian type}, that is, it lies in the rigid tensor subcategory of the category of Chow motives generated by abelian varieties. These facts combining with Theorem \ref{thm2.7} and Corollary \ref{cor2.8}, lead to conjecture that the Chow motive of any projective hyperk\"{a}hler variety is of abelian type. For projective deformations of generalized Kummer varieties, it is expected that their associated Chow motives lies in the rigid tensor subcategory of the category of Chow motives generated by the Chow motives of abelian surfaces.

(ii) Theorem \ref{cor2.8} implies that the Chow motive of any generalized Kummer variety is finite-dimensional in the sense of Kimura and O'Sullivan. See [K05] for more details.

(iii) According to Theorem \ref{thm2.7} and Corollary \ref{cor2.8}, any generalized Kummer variety $K^{[n+1]}$ (of dimension $2n$) admits a Chow-K\"{u}nneth decomposition \cite{Mur93}, that is, there exists a set $\{\pi^{i}\}_{1\leq i\leq 2n}$ of orthogonal idempotents in the correspondence ring $\text{CH}^{2n}(K^{[n+1]}\times K^{[n+1]})_{\mathbb{Q}}$, lifting the set of K\"{u}nneth projectors. Indeed, first note that the right hand side of (*) has a Chow-K\"{u}nneth decomposition, since any abelian variety does. Then we can construct the Chow-K\"{u}nneth decomposition by making use of (*). We will show in a progressing work that for a properly chosen self-dual Chow-K\"{u}nneth decomposition of the right hand side of (*), the so obtained self-dual Chow-K\"{u}nneth decomposition of the generalized Kummer variety is 'multiplicative' in some sense.
\end{remark}

\begin{remark}
Corollary \ref{cor2.8} is the Chow-theoretic lifting of the formula in Theorem 7 of \cite{GS93} at the level of (mixed) Hodge structures.
\end{remark}
\begin{remark}
The results in this section hold over any algebraically closed field of characteristic 0.
\end{remark}

\bigskip

\section{Two applications}
\bigskip

In this section, we provide two applications to Theorem \ref{thm2.7} and Corollary \ref{cor2.8}. The first one is to show the Hodge conjecture holds for generalized Kummer varieties and even for products of generalized Kummer varieties. Other than Theorem \ref{thm2.7}, the other key ingredient is motivation results for Hodge cycles from \cite{Ar06}. The other one is to verify Voevodsky's nilpotence conjecture for 1-cycles on products of generalized Kummer varieties, combining with recent results of this sort on 1-cycles on abelian varieties \cite{Seb13}.

\subsection{Hodge cycles}
First, we recall some notations \cite{An96}\cite{Ar06}. Let $\mathcal{V}$ be an admissible category, i.e, a full subcategory of the category $SP_{k}$ of smooth projective varieties over a field $k$ containing $\text{Spec}k$, $\mathbb{P}^{1}$ and stable under products, disjoint union and connected components. Then for any smooth projective $k$-variety $X$, the $\mathbb{Q}$-algebra $A_{\text{mot}}^{*}(X,\mathcal{V})$ of motivated cycles on $X$ modeled on $\mathcal{V}$ is defined as follows: a cohomology class $\alpha\in A_{\text{mot}}^{*}(X,\mathcal{V})$ if and only if there exists an object $Y$ in $\mathcal{V}$ and algebraic cycle classes $z, w$ on $X\times Y$ such that $\alpha=p_{*}(z\cup *w)$, where $p: X\times Y\rightarrow X$ is the projection and $*$ is the Lefschetz involution with respect to some product polarization. By definition, $A_{\text{mot}}^{*}(X,\mathcal{V})$ contains all algebraic cycles classes on $X$ and coincides assuming Grothendieck's Lefschetz standard conjecture. General formulations give a category $\mathscr{M}_{A}(\mathcal{V})$ of so-called Andr\'{e} motives starting with $\mathcal{V}$ by using $A_{\text{mot}}^{*}(X,\mathcal{V})$. Denote $\mathscr{M}_{A}=\mathscr{M}_{A}(SP_{k})$. Let $\mathscr{M}_{\text{rat}}$ be the category of Chow motives of smooth projective varieties. Then there is a natural functor $\mathscr{M}_{\text{rat}}\rightarrow \mathscr{M}_{A}$. Given $X, Y\in SP_{k}$, $Y$ is said to be motivated by $X$ if the Andr\'{e} motive of $Y$ is isomorphic in $\mathscr{M}_{A}$ to an object in $\mathscr{M}_{A}(\langle X\rangle)$, where $\langle X\rangle$ is the smallest admissible category containing $X$.

The motivation principle in \cite{An96} for Hodge cycles is the following.
\begin{theorem}\label{thm3.1}(\cite{Ar06}, Lemma 4.2)
Suppose that for given two smooth complex projective varieties $X$ and $Y$, $Y$ is motivated by $X$. If the Hodge conjecture holds for $X$ and all of its self-products, then so does for $Y$ and all of its self-products.
\end{theorem}

Now consider the case of products of generalized Kummer varieties.

\begin{lemma}\label{lem3.2}
Let $A_{i} (1\leq i\leq r)$ be abelian surfaces and $K_{i}^{[n_{i}]}$ be the $n_{i}$-th generalized Kummer variety associated to $A_{i}$. Then $K_{1}^{[n_{1}]}\times\cdots\times K_{r}^{[n_{r}]}$ is motivated by $A_{1}\times\cdots\times A_{r}$.
\end{lemma}
\begin{proof}
This an immediate consequence of Theorem \ref{thm2.7} or Corollary \ref{cor2.8} by definition of motivation. Indeed, for single abelian surface $A$, the Chow motive $h(A^{(\lambda)})$ is a direct summand of $h(A^{\lambda})$ for any $\lambda\in P(n)$. Passing to Andr\'{e} motives by the natural functor $\mathscr{M}_{\text{rat}}\rightarrow \mathscr{M}_{A}$, we see that $A\times K^{[n]}$ and hence also $K^{[n]}$ is motivated by $A$. The product case follows from this easily.
\end{proof}

\begin{theorem}\label{thm3.3}
Let $A_{i} (1\leq i\leq r)$ be complex abelian surfaces and $K_{i}^{[n_{i}]}$ be the $n_{i}$-th generalized Kummer variety associated to $A_{i}$. Then the Hodge conjecture holds for the product $K_{1}^{[n_{1}]}\times\cdots\times K_{r}^{[n_{r}]}$.
\end{theorem}
\begin{proof}
This is implies by Theorem \ref{thm3.1}, Lemma \ref{lem3.2} and the fact that the Hodge conjecture holds for arbitrary products of complex abelian surfaces. For the last point, see, for example, \cite{RM08}.
\end{proof}
\begin{remark}
It is well known that Hodge cycles on products of abelian surfaces are represented by Lefschetz classes, i.e, $\mathbb{Q}$-linear combinations of intersections of divisors. We have seen that generalized Kummer varieties are motivated by products of abelian surfaces. However, in general, Hodge cycles on generalized Kummer varieties of dimension at least 4 cannot be represented by Lefschetz classes. In fact, in general, the second Chern class is not Lefschetz. The situation is similar for products of generalized Kummer varieties. For example, for any Kummer K3 surface $S$, the diagonal class on $S\times S$ is not Lefschetz.
\end{remark}

\subsection{Smash-nilpotent 1-cycles}

Recall that for a smooth projective variety $X$ over a field $k$, an algebraic cycle class $z\in\text{CH}^{*}(X)_{\mathbb{Q}}$ is smash-nilpotent if there exists a positive integer $m$ such that the $m$-th exterior product $z^{\times m}=0$ in $\text{CH}^{*}(X^{m})_{\mathbb{Q}}$. It is easy to see that smash-nilpotence defines in fact an adequate equivalence relation on algebraic cycles in the sense of Samuel. Any smash-nilpotent cycle is homologically equivalent to 0 and hence numerically equivalent to 0. Voevodsky \cite{Voe95} conjectures that the converse holds as well. Obviously, this conjecture implies the standard conjecture $D$: homological equivalence coincides with numerical equivalence (with $\mathbb{Q}$-coefficients).

\begin{conjecture}(\cite{Voe95})
Let $X$ be any smooth projective variety over a field. Then any numerically trivial cycles on $X$ is smash-nilpotent.
\end{conjecture}
A basic result on this conjecture due to Voevodsky \cite{Voe95} and Voisin \cite{Voi96} is that any algebraically trivial cycle on a smooth projective variety is smash-nilpotent.

\begin{theorem}\label{thm3.6}(\cite{Seb13})
Let $X$ be any smooth projective variety dominated by some product of curves. Then any numerically trivial 1-cycles on $X$ is smash-nilpotent.
\end{theorem}

Now we show Voevodsky's conjecture holds for 1-cycles on an arbitrary product of generalized Kummer varieties.

We will use the following result which can be shown by the same argument in \cite{Seb13}.

\begin{lemma}\label{lem3.7}
Let $X_{i}\ (1\leq i\leq r)$ be smooth projective varieties. Assume that all numerically trivial 1-cycles on $X_{i}\times X_{j}$ are smash-nilpotent for distinct $i$ and $j$. Then any numerically trivial 1-cycle on $X_{1}\times\cdots\times X_{r}$ is smash-nilpotent.
\end{lemma}

\begin{theorem}\label{thm3.8}
Let $A_{i} (1\leq i\leq r)$ be abelian surfaces and $K_{i}^{[n_{i}]}$ be the $n_{i}$-th generalized Kummer variety associated to $A_{i}$. Then any numerically trivial 1-cycle on the product $K_{1}^{[n_{1}]}\times\cdots\times K_{r}^{[n_{r}]}$ is smash-nilpotent.
\end{theorem}
\begin{proof}
By Lemma \ref{lem3.7}, we may assume that $r=2$. Note that for a surjective morphism of smooth projective varieties $f: X\rightarrow Y$, if any numerically trivial $m$-cycle on $X$ is smash-nilpotent, then so is any numerically trivial $m$-cycle on $Y$.
Then we only have to prove the result for $A_{1}\times K_{1}^{[n_{1}]}\times A_{2}\times K_{2}^{[n_{2}]}$. Now according to Corollary \ref{cor2.8},
$$
h(A_{1}\times K_{1}^{[n_{1}]}\times A_{2}\times K_{2}^{[n_{2}]})\cong\bigoplus_{\lambda\in P(n_{1}),\lambda'\in P(n_{2})}\bigoplus_{\tau\in A_{1}[e(\lambda)],\tau'\in A_{2}[e(\lambda')]}h(A_{1}^{(\lambda)}\times A_{2}^{(\lambda')})\otimes\mathbb{L}^{n_{1}+n_{2}-\ell_{\lambda}-\ell_{\lambda'}}.
$$
Therefore,
$$
\text{CH}_{1}(A_{1}\times K_{1}^{[n_{1}]}\times A_{2}\times K_{2}^{[n_{2}]})_{\mathbb{Q}}\cong\bigoplus_{\ell_{\lambda}+\ell_{\lambda'}\geq n_{1}+n_{2}+1}\bigoplus_{\tau,\tau'}\text{CH}^{n_{1}+n_{2}+\ell_{\lambda}+\ell_{\lambda'}-1}(A_{1}^{(\lambda)}\times A_{2}^{(\lambda')})_{\mathbb{Q}}.
$$
So, any numerically trivial 1-cycle on $A_{1}\times K_{1}^{[n_{1}]}\times A_{2}\times K_{2}^{[n_{2}]}$ is mapped by the above isomorphism to a direct sum of numerically trivial 1-cycles or 0-cycles on $A_{1}^{(\lambda)}\times A_{2}^{(\lambda')}$. However, for any integer $p$, $\text{CH}^{p}(A_{1}^{(\lambda)}\times A_{2}^{(\lambda')})_{\mathbb{Q}}$ is a subgroup of $\text{CH}^{p}(A_{1}^{\lambda}\times A_{2}^{\lambda'})_{\mathbb{Q}}$ via the pullback of the quotient morphism. Now we can draw the conclusion by Theorem \ref{thm3.6} and the fact that numerically trivial 0-cycles are exactly algebraically trivial ones.
\end{proof}

\vspace{1cm}

{\small

\noindent
{\bf Ze Xu}\\ Beijing International Center for Mathematical Research, Peking University, 5 Yiheyuan Road, Haidian District, Beijing 100871, China\\
{\bf Email: xuze@amss.ac.cn}

\end{document}